\documentclass[12pt,twoside]{amsart}
\usepackage{amsmath, amsthm, amscd, amsfonts, amssymb, graphicx, color}
\usepackage[bookmarksnumbered, colorlinks, plainpages]{hyperref}

\theoremstyle{definition}

\theoremstyle{definition}
\theoremstyle{remark}

\newtheorem{thm}{\textbf{Theorem}}[section]
\newtheorem{cor}[thm]{\textbf{Corollary}}

\newtheorem{de}[thm]{\textbf{Definition}}
\newtheorem{ex}[thm]{\textbf{Example}}
\newtheorem{rem}[thm]{\textbf{Remark}}

\numberwithin{equation}{section}

\title[A   new d. for vi on  rnls and    c. th.  it is s. for $(u, v)$-cm]{A   new definition for variational inequalities on  real normed linear spaces and  the case that  it is singelton for $(u, v)$-cocoercive   mappings}
\author{ Ebrahim  Soori }
 \thanks{ \!\!\!\!\!\!\!\! \!\!2010 Mathematics Subject Classification: 90C33; 47H10.  \\ E-mail address: sori.e@lu.ac.ir, sori.ebrahim@yahoo.com
  \\Tel: +98 9188521850 (E. Soori)}

\begin{document}
\begin{large}


\maketitle


\begin{center}
 \begin{normalsize}
    Department  of Mathematics, Lorestan University, Khoramabad, Lorestan, Iran.
 \end{normalsize}
 \end{center}
\begin{abstract}

\begin{normalsize}
Let $C$ be a nonempty closed   convex subset of a  Banach space $E$. In this paper we introduce  a new definition for variational inequality  $V I (C, B)$ on $E$   that   generalizes the analogue   definition on Hilbert spaces. We    generalize $(u, v)$-cocoercive mappings and $v$-strongly  monotone mappings   from Hilbert spaces to  Banach spaces. Then we      prove  the generalized variational inequality   $V I (C, B)$ is singleton for   $(u, v)$-cocoercive mappings     under appropriate assumptions on Banach spaces that extends and improves [S. Saeidi, Comments on relaxed $(u, v)$-cocoercive mappings. Int. J. Nonlinear Anal. Appl. 1 (2010) No. 1, 54-57].
\end{normalsize}
\end{abstract}
\begin{normalsize}
   \textbf{keywords}:   Fixed point; Nonexpansive mapping; $(u, v)$-cocoercive;  Duality mapping;     sunny nonexpansive retraction.
   \end{normalsize}

\section{ Introduction}
Let $C$ be a nonempty closed and convex subset of a Banach space $E$ and
$E^{*}$ be the dual space of $E$. Let $\langle.,.\rangle$   denote the pairing between $E$ and $E^{*}$. The
normalized duality mapping $J: E \rightarrow E^{*}$
is defined by
\begin{align*}
    J(x)=\{f \in E^{*}: \langle x, f \rangle= \|x\|^{2}=\|f\|^{2} \}
\end{align*}
for all $x \in E$ (Similarly, the  mapping $J$ has  defined for normed spaces in \cite{Ag}).  Let  $U = \{x \in E : \|x\| = 1\}$.  A Banach space $E$ is said to be smooth if for each $x \in U$,
there exists a unique functional $j_{x} \in E^{*}$ such that $\langle x, j_{x}\rangle = \|x\|$ and $\|j_{x}\| = 1$ \cite{Ag}.

Let $C$ be a nonempty closed and convex subset of a Banach space $E$.  A mapping $ T$ of $ C $ into itself is called nonexpansive if $\|Tx - Ty\| \leq \|x - y\|,$ for all $x, y \in C$ and a mapping $f$ is an $\alpha$-contraction on $E $ if  $ \|f (x) -f (y)\| \leq \alpha \|x - y\|, \;x, y \in E$  such that  \\$0 \leq\alpha < 1$.  A mapping $T:C\rightarrow C$ is called Lipschitzian if  there exists a nonnegative number $k$ such that\\ $ \|Tx-Ty\|\leq k\|x-y\| \quad \text{for all} \;\;x,y\in C$.

 Let $C$ be a nonempty closed convex subset of a real Hilbert space $ H $. Let $B : C \rightarrow H $ be a nonlinear map.   Let $ P_{C}$ be the projection of $ H$ onto $ C$. Then the projection operator  $P_{C}$ assigns to each $ x\in H$,  the unique point $  P_{C} x \in C$ satisfying the property
\begin{equation*}
    \|x-P_{C} x\|=\min_{y\in C}\|x-y\|.\qquad\qquad\qquad\qquad\qquad\qquad
\end{equation*}
 The classical variational inequality problem, denoted by $V I (C, B)$ is to find $u \in C $ such that
 \begin{equation}\label{23}
    \langle Bu, v - u\rangle\geq 0,
 \end{equation}
 for all $v \in C$ (see \cite{sa}). For a given $z \in H$, $u \in C$ satisfies the inequality
 \begin{equation}\label{w}
    \langle u - z, v - u\rangle\geq 0,\quad  ( v \in C),
 \end{equation}
 if and only if $u = P_{C} z$. Therefore
\begin{align*}
u \in V I(C,B) \Longleftrightarrow u = P_{C}(u -\lambda Bu),
\end{align*}
where  $\lambda > 0 $ is a constant(see \cite{sa}).  It is known that the projection operator  $P_{C}$ is nonexpansive. It is also known that  $P_{C}$  satisfies
 \begin{equation}\label{24}
    \langle x - y, P_{C} x - P_{C} y\rangle \geq \| P_{C} x - P_{C} y\|^{2},
 \end{equation}
 for $x,y \in H $.

 Let $C$ be a nonempty closed convex subset of a real Hilbert space $ H $, recall   the following definitions (see \cite{sa}):\\
 \begin{enumerate}
   \item [(i)] $B$ is called $v$-strongly monotone if
 $$\langle Bx - By\;,\; x - y\rangle\geq v\|x - y\|^{2}\qquad \text{for\: all}\quad x, y \in C,\qquad\qquad\qquad\qquad$$
 for a constant $v > 0$.
   \item [(ii)] $B$ is said to be relaxed $(u, v)$-cocoercive, if there exist two constants $u, v > 0$ such that
 \begin{equation*}
    \langle Bx - By, x - y\rangle \geq (-u)\|Bx - By\|^{2}+v\|x - y\|^{2},
 \end{equation*}
   for  all  $x, y \in C$. For $u = 0$, $B$ is $v$-strongly monotone. This class of maps is more general than the class of strongly monotone maps. Clearly, every  $v$-strongly monotone map is a relaxed $(u, v)$-cocoercive map.
 \end{enumerate}

Let $C$ be a nonempty closed convex subset of a Banach space $E$. In this paper we introduce    a definition for variational inequality    on Banach spaces  that   generalize the analogue   definition  on  Hilbert spaces. Then we      prove the  variational inequality is singleton for
   $(u, v)$-cocoercive mappings  under appropriate assumptions.

\section{preliminaries}

Let $E$ be a real Banach space with its dual $E^{*}$.  A Banach
space $E$ is said to be strictly convex if
\begin{align*}
\Vert x\Vert=\Vert y\Vert=1, \;\; x\neq y \Rightarrow \Vert \frac{x+y}{2}\Vert <1.
\end{align*}

Let $C$ be a nonempty subset of a normed space $E$ and let $x \in E$. An element
$y_{0} \in C$ is said to be a best approximation to $x$ if
$\|x - y_{0}\| = d(x,C)$,
where
\begin{equation}\label{pcd}
    d(x,C) = \inf _{y\in C} \|x - y\|.
\end{equation}
  The number $d(x,C)$ is called the distance from $x$
to $C$ or the error in approximating $x$ by $C$.

The (possibly empty) set of all best approximations from $x$ to $C$ is denoted
by
$P_{C}(x) = \{y \in C : \|x - y\| = d(x,C)\}$.
This defines a mapping $P_{C}$ from $E$ into $2^{C}$ and is called the metric projection
onto $C$. The metric projection mapping is also known as the nearest point
projection mapping, proximity mapping, and best approximation operator.

Let $ C $ be a nonempty closed subset of a Banach
space $E$. Then a mapping  $Q : E \rightarrow C$  is said to be sunny if
$Q(Qx + t (x - Qx)) = Qx$,   $\forall x \in E,\; \forall  t\geq0 $.
A mapping   $Q : E \rightarrow C$  is said to be a retraction or a projection if $Qx = x$,  $\forall x \in C$. If $E$ is smooth
then the sunny nonexpansive retraction of $E$ onto $ C $ is uniquely decided (see \cite{ta1}). Then, if $E$ is
a smooth Banach space, the sunny nonexpansive retraction of $E$ onto $ C $ is denoted by  $Q_{C}  $.  Let $ C $
be a nonempty closed subset of a Banach space $E$. Then a subset $ C $ is said to be a nonexpansive
retract (resp. sunny nonexpansive retract) if there exists a nonexpansive retraction (resp. sunny
nonexpansive retraction) of $E$ onto $ C $ (see \cite{br2, br3}).
Let $ C $ be
a nonempty closed convex subset of a smooth, reflexive, and strictly convex Banach space $ E $.  Let  $Q_{C}  $  be  the
sunny nonexpansive retraction.  Then we have
\begin{align}\label{qc}
x_{0}=Q_{C}x    \Longleftrightarrow \langle  x-x_{0}\,,\,J(x_{0}-y)\rangle    \geq 0,
\end{align}
 for each $y \in C$.
We have  $P_{C}= Q_{C} $ in a Hilbert space (see \cite{it}).

\section{main results}

  First, we  introduce the following   new definition:
\begin{de}
Let   $C$ be a nonempty  closed convex subset of a real normed linear space  $E$ and\\ $B : C \rightarrow E $ be a nonlinear map. $B$ is said to be relaxed $(u, v)$-cocoercive, if there exist two constants $u, v > 0$ such that
 \begin{equation*}
    \langle Bx - By, j(x - y)\rangle \geq (-u)\|Bx - By\|^{2}+v\|x - y\|^{2},
 \end{equation*}
   for  all $x,y\in C$ and $j(x - y)\in J(x-y)$.
\end{de}
\begin{ex}\label{ex2}
  Let $C$ be a nonempty closed convex subset of a real Hilbert space $ H $, it is well-known  that
 $B : C \rightarrow H $ is said to be relaxed $(u, v)$-cocoercive, if there exist two constants $u, v > 0$ such that
 \begin{equation*}
    \langle Bx - By, x - y\rangle \geq (-u)\|Bx - By\|^{2}+v\|x - y\|^{2},
 \end{equation*}
 for  all  $x, y \in C$.  By Example 2.4.2 in \cite{Ag}, in a Hilbert space $ H $, the normalized duality mapping is the
identity.     Then $J(x-y)=\lbrace x-y\rbrace  $. Therefore, the above definition  extends   the definition of  relaxed $(u, v)$-cocoercive mappings, from real Hilbert spaces   to real  normed linear  spaces.
  \end{ex}

 Let us to define $v$-strongly monotone mappings on real normed linear spaces, too.
 \begin{de}
Let   $C$ be a nonempty  closed convex subset of a real normed linear space  $E$ and\\ $B : C \rightarrow E $ be a nonlinear map. $B$  is called $v$-strongly monotone   if there exists     a constant $v > 0$ such that
 $$\langle Bx - By\;,\; j(x - y)\rangle\geq v\|x - y\|^{2},$$
   for  all $x,y\in C$ and $j(x - y)\in J(x-y)$.
\end{de}
\begin{ex}\label{ex3}
  Let $C$ be a nonempty closed convex subset of a real Hilbert space $ H $, it is well-known, too,  that
 $B : C \rightarrow H $ is said to be $v$-strongly monotone, if there exists  a constant $v > 0$ such that
 \begin{equation*}
    \langle Bx - By, x - y\rangle\geq v\|x - y\|^{2},
 \end{equation*}
 for  all  $x, y \in C$.
 Since $ H $ is  a Hilbert space,\\ $J(x-y)=\lbrace x-y\rbrace$. Therefore, the above definition  extends   the definition of  $v$-strongly monotone mappings, from real  Hilbert spaces to  real  normed linear spaces.
 \end{ex}

\begin{ex}\label{zendegi}
 Let $C$ be a nonempty closed convex subset of a real  Banach space $E$. Let $T$ be an $\alpha$-contraction of  $C$ into itself. Putting  $B=I-T$, we have
\begin{align*}
    \langle Bx - By,& j(x - y)\rangle\\ =& \langle (I-T)x - (I-T)y, j(x - y)\rangle\\=&
 \langle (x-y) - (Tx-Ty), j(x - y)\rangle\\=&   \langle x-y , j(x - y)\rangle- \langle Tx-Ty , j(x - y)\rangle\\\geq &
 \langle x-y , j(x - y)\rangle - \| Tx-Ty \| \|j(x - y)\|\\\geq &
 \| x-y \|^{2} - \| Tx-Ty \| \|x - y\|\\\geq &  \| x-y \|^{2} -\alpha \| x-y \|^{2}=(1-\alpha)\| x-y \|^{2}.
\end{align*}
Hence $B : C \rightarrow E $ is a $(1-\alpha)$-strongly monotone mapping,  therefore $B$ is a relaxed $(u, (1-\alpha))$-cocoercive mapping on $E$ for each  $u >0 $.
\end{ex}
 Now, we introduce   the following new definition that generalizes  the classical variational inequality problem  \ref{23}.
\begin{de}
 Let $ E $ be a real normed linear space.  Let $C$ be a nonempty  closed convex subset of   $E$. Let \\$B : C \rightarrow E $ be a nonlinear map.    The      classical variational inequality problem $V I (C, B)$  is to find $u \in C $ such that
 \begin{equation}\label{vi}
    \langle Bu\;,\; j( v - u) \rangle \geq 0,
 \end{equation}
 for all $v \in C$ and $j(v - u) \in J(v - u)$.
\end{de}
\begin{ex}\label{ex4}
 Let $C$ be a nonempty closed convex subset of a real Hilbert space $ H $     and
 $B : C \rightarrow H $   be a relaxed $(u, v)$-cocoercive mapping. Since $ H $ is a   Hilbert space,   $j(v - u)=\lbrace v-u\rbrace  $. Therefore,   \ref{vi}  generalizes    \ref{23} from   real  Hilbert spaces to real normed linear spaces.
 \end{ex}
 \begin{rem}\label{cheb}
Let $ C $ be a nonempty closed convex subset of a smooth, reflexive, and strictly convex Banach space $ E $.  Let  $Q_{C}  $  be  the
sunny nonexpansive retraction. By \eqref{qc}, we have
\begin{align}\label{aga}
u \in V I(C,B) \Longleftrightarrow u = Q_{C}(u -\lambda Bu).
\end{align}
\end{rem}
\begin{thm}\label{jmen} Let $E$ be a Banach space, for all $x,y \in E$, we have
   $$ \langle x-y\;,\; j(x-y)\rangle\leq \langle x-y\;,\;  x^{*}-y ^{*}\rangle+4\|x\|\|y\|,$$
   for all $x^{*}\in J(x),  y ^{*} \in J(y), j(x-y) \in J(x-y)$.
\end{thm}
\begin{proof}
   Let $x= y$, obviously the inequality holds. Let  $x^{*}\in J(x),  y ^{*} \in J(y) $ and $x\neq y$. As in the proof of Theorem 4.2.4 in \cite{tn}, we have
    \begin{align*}
        \langle x-y\;,\;  x^{*}-&y ^{*}\rangle \\ \geq & (\|x\|-\|y\|)^{2}\\&+ (\|x\|+\|y\|)(\|x\|+\|y\|-\|x+y\|).
    \end{align*}
    Hence, we have
    \begin{align*}
        \langle x-y\;,\;  x^{*}-&y ^{*}\rangle \\ \geq & (\|x\|-\|y\|)^{2}\\&+ (\|x\|+\|y\|)(\|x\|+\|y\|-\|x+y\|)\\=& (\|x\|-\|y\|)^{2}\\&+ (\|x\|+\|y\|)^{2}-\|x+y\| (\|x\|+\|y\|)\\ \geq &
        (\|x\|-\|y\|)^{2}\\&+ \|x-y\|^{2}- (\|x\|+\|y\|)^{2}
         \\ = & \|x-y\|^{2}-4\|x\|\|y\|\\=&
        \langle x-y\;,\; j(x-y)\rangle -4\|x\|\|y\|,
        \end{align*}
        therefore,
           $$ \langle x-y\;,\; j(x-y)\rangle\leq \langle x-y\;,\;  x^{*}-y ^{*}\rangle+4\|x\|\|y\|.$$
\end{proof}

Now we are ready to prove the main theorem:
\begin{thm}\label{thm1}
Let $ C $ be a nonempty closed convex subset of a smooth, reflexive, and strictly convex Banach space $ E $.   Suppose that $\mu>0 $,  and $v > u\mu^{2}+5\mu$. Let  $B: C \rightarrow E$
be a relaxed $(u, v)$-cocoercive,  $ \mu$-Lipschitzian mapping.  Let  $Q_{C}  $  be  the
sunny nonexpansive retraction  from $E$ onto $ C $.     Then $V I(C,B)$ is singleton.
\end{thm}
\begin{proof}
Let $\lambda$ be a real number such that
\begin{equation}
    0<\lambda<\frac{v-u\mu^{2}-5\mu}{\mu^{2}}, \quad \lambda   \mu^{2}[\frac{v-u\mu^{2}-5\mu}{\mu^{2}}- \lambda ]<1.
\end{equation}
    Then, by Theorem \ref{jmen},  for every $x, y \in C$, we have
    \begin{align*}
        \|Q_{C}&(I -\lambda B)x- Q_{C}(I -\lambda B)y\|^{2}\\ \leq &
        \|(I -\lambda B)x- (I -\lambda B)y\|^{2}\\ = &
        \|(x-y)-  \lambda (Bx- By)\|^{2}\\ = &
        \|j[(x-y)-  \lambda (Bx- By)]\|^{2}\\ = &
        \langle (x-y)-  \lambda (Bx- By),j[(x-y)\\&-  \lambda (Bx- By)]\rangle \\ \leq &
        \langle x-y -\lambda (Bx- By)  \;,\; j(x-y)\\&-\lambda j(Bx- By)\rangle\\&+4\lambda \|x-y \| \|Bx- By \|\\=&
        \langle x-y   \;,\; j(x-y)\rangle\\&-\lambda \langle  Bx- By   \;,\; j(x-y)\rangle \\&+\lambda \langle y-x  \;,\;   j(Bx- By)\rangle \\&+ \lambda ^{2} \langle  (Bx- By)  \;,\;  j(Bx- By)\rangle \\&+4\lambda \|x-y \| \|Bx- By \|\\\leq &
        \|x-y\|^{2}+\lambda u  \|Bx- By\|^{2}- \lambda v \|x-y\|^{2}\\&+ \lambda^{2} \|Bx- By\|^{2}+5\lambda \|x-y \| \|Bx- By \|\\ \leq &
        \|x-y\|^{2}+\lambda u  \mu^{2}\|x- y\|^{2}- \lambda v \|x-y\|^{2}\\&+ \lambda^{2}\mu^{2} \|x-y\|^{2}+5\lambda \mu \|x-y \| ^{2}\\ \leq &
        \Big(1+\lambda u  \mu^{2}- \lambda v+ \lambda^{2}\mu^{2} +5\lambda \mu \Big)\|x-y\| ^{2}\\\leq &
        \Big(1-\lambda   \mu^{2}[\frac{v-u\mu^{2}-5\mu}{\mu^{2}}- \lambda ] \Big)\|x-y\| ^{2}
    \end{align*}
    Now, since $1-\lambda   \mu^{2}[\frac{v-u\mu^{2}-5\mu}{\mu^{2}}- \lambda ] < 1$, the mapping \\$Q_{C}(I -\lambda B) : C \rightarrow C$ is a contraction
and Banach's Contraction Mapping Principle guarantees that it has a unique fixed
point  $  u$; i.e., $Q_{C}(I -\lambda B)u = u$, which is the unique solution of $V I(C,B)$ by  \ref{aga}.
\end{proof}
Since $v$-strongly monotone mappings are relaxed $(u, v)$-cocoercive, we conclude  the following theorem.
\begin{thm}\label{thm2}
Let $ C $ be a nonempty closed convex subset of a smooth, reflexive, and strictly convex Banach space $ E $.  Let  $Q_{C}  $  be  the
sunny nonexpansive retraction.  Suppose that $\mu>0 $,  and $v > u\mu^{2}+5\mu$. Let  $B: C \rightarrow E$
be a  $v$-strongly monotone,  $ \mu$-Lipschitzian mapping.     Then $V I(C,B)$ is singleton.
\end{thm}

 We can conclude  Proposition 2 in \cite{sa} for $v > u\mu^{2}+5\mu$, as follows:
\begin{cor} Let $C$ be a nonempty closed convex subset of a Hilbert space  $H$ and let  $B : C \rightarrow H$
be a relaxed $(u, v)$-cocoercive and  $ 0<\mu$-Lipschitzian mapping such that $v > u\mu^{2}+5\mu$. Then
$V I(C,B)$ is singleton.

\end{cor}
\begin{rem}
S.  Saeidi, in the proof of  Proposition 2  in  \cite{sa}          proves  that
\begin{align*}
 \|P_{C}&(I -  s A)x- P_{C}(I -s A)y\|^{2} \\& \leq  \Big(1-s   \mu^{2}[\frac{2(r-\gamma \mu^{2})}{\mu^{2}}- s ] \Big)\|x-y\| ^{2}
\end{align*}
when $ 0< s<\frac{2(r-\gamma \mu^{2})}{\mu^{2}}$ and $ r>\gamma \mu^{2}$. Putting \\$ r=\gamma=s=1 $ and $ \mu=\frac{1}{10} $ we have \\$\Big(1-s   \mu^{2}[\frac{2(r-\gamma \mu^{2})}{\mu^{2}}- s ] \Big)<0  $ that is a contradiction. We correct  this contradiction in the proof of theorem \ref{thm1}.

\end{rem}

We can conclude  Proposition 3 in \cite{sa} for $v > 5\mu$, as follows:
\begin{cor} Let $C$ be a nonempty closed convex subset of a Hilbert space  $H$ and let  $B : C \rightarrow H$
be a  $v$-strongly monotone and $ 0<\mu$-Lipschitzian mapping such that  $v > 5\mu$. Then
$V I(C,B)$ is singleton.

\end{cor}

\end{large}


\bigskip
\bigskip



\begin{thebibliography}{20}
\bibitem{Ag} R. P. Agarwal,   D. Oregan and D. R. Sahu, Fixed point theory for Lipschitzian-type mappings with applications, in: Topological Fixed Point Theory and its Applications, vol. 6, Springer, New York, 2009.
\bibitem{br} R.E. Bruck, Jr., Nonexpansive projections on subsets of Banach spaces, Pacific. J. of
Math., 47 (1973), 341-356.

\bibitem{br2} R.E. Bruck, Nonexpansive retract of Banach spaces, Bull. Amer. Math. Soc. 76 (1970) 384-386.
\bibitem{br3} R.E. Bruck, Properties of fixed-point sets of nonexpansive mapping in Banach spaces, Trans. Amer. Math. Soc. 179
(1973) 251-262.


\bibitem{it} T. Ibarakia, W. Takahashi,  A new projection and convergence theorems for the
projections in Banach spaces, Journal of Approximation Theory 149 (2007) 1 - 14.
     \bibitem{sa} S. Saeidi, Comments on relaxed $(\gamma, r)$-cocoercive mappings. Int. J. Nonlinear Anal. Appl.
1 (2010) No.1, 54-57.
\bibitem{ta1} W. Takahashi, Convergence theorems for nonlinear projections in Banach spaces (in Japanese), RIMS Kokyuroku,
vol. 1396, 2004, pp. 49-59.
\bibitem{tn} W. Takahashi, Nonlinear Functional Analysis: Fixed Point Theory and its Applications, Yokohama Publishers, Yokohama, 2000.

\end{thebibliography}
\end{document}